\documentclass[11pt]{article}

\usepackage[a4paper,margin=1in]{geometry}
\usepackage{amsmath,amssymb,amsthm,mathtools}
\usepackage{booktabs}
\usepackage{microtype}
\usepackage[hidelinks]{hyperref}
\usepackage{lmodern}

\allowdisplaybreaks
\newtheorem{theorem}{Theorem}[section]
\newtheorem{proposition}[theorem]{Proposition}
\newtheorem{lemma}[theorem]{Lemma}
\newtheorem{corollary}[theorem]{Corollary}
\theoremstyle{definition}
\newtheorem{definition}[theorem]{Definition}

\newcommand{\C}{\mathbb C}
\newcommand{\R}{\mathbb R}
\newcommand{\card}{\operatorname{card}}
\newcommand{\Spec}{\operatorname{Spec}}
\newcommand{\spn}{\operatorname{span}}
\newcommand{\ip}[2]{\left\langle #1,#2\right\rangle}
\newcommand{\ray}[1]{[#1]}

\hypersetup{
  pdftitle={A Quantum Latin Square of Order Six with Cardinality 29},
  pdfauthor={Aishwarya P. Das and Durgesh Kumar},
  pdfsubject={An exact construction completing the cardinality spectrum for quantum Latin squares of order six},
  pdfkeywords={quantum Latin square, cardinality, punctured orthonormal array, order six}
}

\title{A Quantum Latin Square of Order Six\\
with Cardinality 29}
\author{Aishwarya P. Das\thanks{Email: \href{mailto:aishwarya@diraclabs.com}{aishwarya@diraclabs.com}}
\quad and \quad
Durgesh Kumar\thanks{Corresponding author. Email: \href{mailto:durgesh@diraclabs.com}{durgesh@diraclabs.com}}\\[-1mm]
\small Dirac Labs, Madison, Wisconsin, USA}
\date{August 2026}

\begin{document}
\maketitle

\begin{abstract}
We construct an explicit real quantum Latin square of order six with
cardinality $29$, the last unresolved value in the order-six spectrum.  We
first construct a punctured $6\times6$ array in $\R^5$ whose punctured rows
and columns are orthonormal bases, and then adjoin a common diagonal vector in
an orthogonal one-dimensional summand.  The six diagonal entries lie on one
ray.  Of the thirty off-diagonal entries, two rays occur twice and the other
twenty-six occur once; exact support and coordinate-ratio comparisons
establish this count.  Together with the known constructions, this closes the
cardinality spectrum of quantum Latin squares of order $6$.
\end{abstract}


\section{Introduction}

Quantum Latin squares were introduced by Musto and Vicary as vector-valued
analogues of ordinary Latin squares~\cite{MustoVicary2016}.  In an ordinary
Latin square the entries are symbols, whereas in a quantum Latin square they
are unit vectors, with every row and column forming an orthonormal basis.
Quantum Latin squares are related to unitary error bases and to more general
quantum combinatorial designs~\cite{Goyeneche2018,MustoVicary2019}.

Order six is especially prominent because of Euler's problem of the
thirty-six officers~\cite{Euler1782} and its genuinely quantum
solution~\cite{Rather2022}.  A separate line of work studies the
\emph{cardinality} of a quantum Latin square, meaning the number of entry
vectors after global phases have been identified~\cite{Paczos2021}.

The possible cardinalities have been studied in several recent
papers~\cite{ZhangWangJi2026,ZangEtAl2025,ZhangCao2026,ZhangJi2026}.
For order six, Xu constructed examples with cardinalities $13$, $15$, and
$17$~\cite{Xu131517}.  Zhang, Lv, and Cao's Table~6 lists the six values
$23,25,27,29,32$, and $35$ as uncertain~\cite{ZhangLvCao2026}.  The current
version of Xu constructs $23,25,27,32$, and $35$ and identifies $29$ as the
sole remaining undetermined value~\cite[Theorems~3--7 and
Corollary~1]{Xu2026}.  In this paper, we construct that example.

\begin{theorem}\label{thm:main}
There exists a real quantum Latin square of order six with cardinality $29$.
\end{theorem}

The idea of the construction is as follows.  We place the same vector in all
six diagonal cells and construct the remaining entries in its orthogonal
complement $\R^5$.  Each punctured row and column must then form an
orthonormal basis.  Since the common diagonal contributes one ray, the thirty
off-diagonal entries must determine twenty-eight rays.  We achieve this by
making two off-diagonal rays occur twice and all others once.

Our construction differs from Xu's Hadamard-product constructions and from
the linked $\C^4\oplus\C^2$ family in~\cite{Xu2026}.  We obtain the punctured
array through a sequence of orthogonal changes of basis in two-dimensional
subspaces.  Most of these changes are local, but one compatibility condition
links them and accounts for the choice of rational parameters below.  In
Section~\ref{sec:construction}, we construct the array and prove
its orthogonality.  In Appendix~\ref{app:separation}, we compare supports and
coordinate ratios to verify the ray count.

\section{Punctured orthonormal arrays}\label{sec:preliminaries}

Write $[n]=\{0,1,\ldots,n-1\}$.  For a nonzero vector $x$, write
\[
\ray{x}=\{\lambda x:\lambda\in\C\setminus\{0\}\}.
\]
If $x$ and $y$ are unit vectors, then $\ray{x}=\ray{y}$ precisely when
$y=e^{i\theta}x$ for some $\theta\in\R$.  Thus equality of rays identifies
exactly the global phase of a unit vector.

\begin{definition}
A \emph{quantum Latin square of order $n$}, denoted $\operatorname{QLS}(n)$,
is an array
\[
\Phi=(\phi_{ij})_{i,j\in[n]},\qquad \phi_{ij}\in\C^n,
\]
of unit vectors such that every row and every column is an orthonormal basis
of $\C^n$.  Its cardinality is
\[
\card(\Phi)=\left|\{\ray{\phi_{ij}}:i,j\in[n]\}\right|.
\]
We call $\Phi$ \emph{real} if all its entries lie in
$\R^n\subset\C^n$, and write
\[
\Spec(\operatorname{QLS}(n))
=\{\card(\Phi):\Phi\text{ is a }\operatorname{QLS}(n)\}.
\]
\end{definition}

If the same vector is placed in every diagonal cell of a quantum Latin
square, then all off-diagonal entries lie in its orthogonal complement.
Removing the diagonal therefore leads to the following structure.

\begin{definition}
A \emph{punctured orthonormal array of order $n$} in $\C^{n-1}$ is a family
\[
V=(v_{ij})_{i,j\in[n],\ i\ne j}
\]
of unit vectors such that every punctured row $(v_{ij})_{j\ne i}$ and every
punctured column $(v_{ij})_{i\ne j}$ is an orthonormal basis of $\C^{n-1}$.
The same terminology will be used over $\R^{n-1}$, regarded as the standard
real subspace of $\C^{n-1}$.
\end{definition}

The next proposition reverses this reduction by adjoining a common diagonal
vector.

\begin{proposition}[Diagonal extension]\label{prop:diagonal-extension}
Let $V$ be a punctured orthonormal array of order $n$ in $\C^{n-1}$.
Identify $\C^n$ with the orthogonal direct sum $\C\oplus\C^{n-1}$, let
$d=(1,0)$, and define
\[
\phi_{ij}=\begin{cases}
d,&i=j,\\
(0,v_{ij}),&i\ne j.
\end{cases}
\]
Then $\Phi=(\phi_{ij})$ is a $\operatorname{QLS}(n)$, and
\[
\card(\Phi)=1+\left|\{\ray{v_{ij}}:i\ne j\}\right|.
\]
\end{proposition}

\begin{proof}
Each row and column consists of the unit vector $d$ in the first summand and
an orthonormal basis of the second, so it is an orthonormal basis of $\C^n$.
Moreover, $(0,v)$ and $(0,w)$ determine the same ray in $\C^n$ exactly when
$v$ and $w$ determine the same ray in $\C^{n-1}$.  The ray of $d$ is distinct
from every ray represented in the second summand.  This proves the
cardinality formula.
\end{proof}

The following lemma allows us to recover one punctured column after all the
others have been checked.

\begin{lemma}[One missing column]\label{lem:missing-column}
Let $V=(v_{ij})_{i\ne j}$ be an order-$n$ array of unit vectors in
$\C^{n-1}$.  If every punctured row and all but one punctured column are
orthonormal bases, then the remaining punctured column is also an
orthonormal basis.
\end{lemma}

\begin{proof}
Let $j_0$ denote the column which has not been checked.  Since each punctured
row is an orthonormal basis, summing all the rank-one frame operators row by
row gives $nI_{n-1}$.  If the same sum is grouped by columns, the $n-1$ known
column bases contribute $(n-1)I_{n-1}$.  The frame operator of column $j_0$
is therefore $I_{n-1}$.  This column has $n-1$ vectors in dimension $n-1$,
so its square synthesis matrix is unitary.  Hence it is an orthonormal basis.
\end{proof}

\section{The construction}\label{sec:construction}

\subsection{The off-diagonal array}

Let $e_0,e_1,e_2,e_3,e_4$ be the standard basis of $\R^5$.  We now construct
the off-diagonal array.  Its entries have the following form, with the
remaining symbols defined in the following subsections.
\begin{equation}\label{eq:V29}
V_{29}=
\begin{pmatrix}
- & e_0 & e_1 & e_2 & e_3 & e_4\\
v_{10} & - & v_{12} & v_{13} & Y & v_{15}\\
v_{20} & v_{21} & - & e_0 & P & R\\
v_{30} & v_{31} & U & - & Q & H\\
v_{40} & v_{41} & v_{42} & v_{43} & - & v_{45}\\
v_{50} & v_{51} & Y & v_{53} & Z & -
\end{pmatrix}.
\end{equation}
The two occurrences each of $e_0$ and $Y$ give two off-diagonal rays that
occur twice.  We shall show in Appendix~\ref{app:separation} that no other
off-diagonal rays coincide.  It remains to define the other entries
in~\eqref{eq:V29}.

The construction is based on replacing orthonormal pairs by other orthonormal
pairs in the same two-dimensional subspaces.  Most of these choices are
independent.  However, in column~$0$ we need one vector in
$\spn\{v_{13},v_{43}\}$ to be perpendicular to both $v_{20}$ and $v_{30}$.
For this to be possible, the projections of $v_{20}$ and $v_{30}$ onto this
subspace must be linearly dependent.  This condition is expressed by the
scalar identity $\Delta=0$ in~\eqref{eq:Delta}, and it accounts for the
particular rational parameters used below.

\subsection{Seed bases}

Set
\begin{equation}\label{eq:parameters}
\begin{array}{c@{\qquad}c@{\qquad}c}
(\alpha,\beta)=\left(\frac{80}{89},\frac{39}{89}\right),&
(\varepsilon,\varphi)=\left(\frac5{13},\frac{12}{13}\right),&
(\gamma,\delta)=\left(\frac{55}{73},\frac{48}{73}\right),\\[2mm]
(\mu,\nu)=\left(\frac{36}{85},\frac{77}{85}\right),&
(\xi,\eta)=\left(\frac{84}{85},\frac{13}{85}\right).
\end{array}
\end{equation}
Each pair in~\eqref{eq:parameters} lies on the unit circle.  We use the first
three pairs to define bases in coordinate planes, and the final two pairs to
define later changes of basis.  These values were found by using the standard
rational parametrization
\[
(x,y)=\left(\frac{1-t^2}{1+t^2},\frac{2t}{1+t^2}\right)
\]
and imposing the condition $\Delta=0$.  The search is used only to obtain the
parameters; all identities needed in the proof are verified exactly below.

Define
\[
N=\sqrt{\varphi^2+\alpha^2\varepsilon^2},
\qquad
\kappa=\frac{\alpha\varepsilon}{\varphi},
\]
so that $N=4\sqrt{481}/89$.  With these constants, set
\begin{align}
Y&=\alpha e_0+\beta e_2,&
Z&=\beta e_0-\alpha e_2,\label{eq:YZ}\\
P&=\gamma e_1+\delta e_4,&
Q&=\delta e_1-\gamma e_4,\label{eq:PQ}\\
R&=\varepsilon e_2+\varphi e_3,&
S&=\varphi e_2-\varepsilon e_3,\label{eq:RS}\\
U&=\frac{\beta\varphi e_0-\alpha\varphi e_2
            +\alpha\varepsilon e_3}{N},\label{eq:U}\\
W&=\frac{-\alpha\beta\varepsilon e_0+\alpha^2\varepsilon e_2
            +\varphi e_3}{N},\label{eq:W}\\
H&=\frac{\alpha e_0+\beta\varphi^2 e_2
            -\beta\varepsilon\varphi e_3}{N}.\label{eq:H}
\end{align}
Direct calculation shows that $U$ is a unit vector perpendicular to both $Y$
and $R$ in $\spn\{e_0,e_2,e_3\}$.  Moreover, $(Y,U,W)$ and $(R,U,H)$ are
orthonormal bases of this space.  The pairs $(Y,Z)$, $(P,Q)$, and $(R,S)$ are
orthonormal bases of their respective coordinate planes.  We shall also use
\begin{equation}\label{eq:key-WU}
W+\kappa U=\frac{N}{\varphi}e_3.
\end{equation}

\subsection{Orthogonal completions}

We first make two changes of basis in orthogonal planes.  Define
\begin{align}
v_{12}&=\mu W+\nu e_4,&
v_{42}&=\nu W-\mu e_4,\label{eq:v1242}\\
v_{15}&=\xi U-\eta e_1,&
v_{45}&=\eta U+\xi e_1.\label{eq:v1545}
\end{align}
The pairs $(W,e_4)$ and $(U,e_1)$ are orthonormal and their spans are
orthogonal.  It follows that $(v_{12},v_{42})$ and $(v_{15},v_{45})$ are
orthonormal pairs whose spans are also orthogonal.  We next define
\begin{align}
r_{13}&=\sqrt{\eta^2+\kappa^2\nu^2},&
r_{43}&=\sqrt{\xi^2+\kappa^2\mu^2},\label{eq:r1343}\\
v_{13}&=\frac{\eta v_{42}+\kappa\nu v_{45}}{r_{13}},&
v_{10}&=\frac{\kappa\nu v_{42}-\eta v_{45}}{r_{13}},\label{eq:v1310}\\
v_{43}&=\frac{\xi v_{12}+\kappa\mu v_{15}}{r_{43}},&
T&=\frac{\kappa\mu v_{12}-\xi v_{15}}{r_{43}}.\label{eq:v43T}
\end{align}
These are orthogonal changes of basis, and hence $(v_{13},v_{10})$ and
$(v_{43},T)$ are orthonormal pairs.  The signs in the definitions have been
chosen so that the numerators of $v_{13}$ and $v_{43}$ both contain the term
$W+\kappa U$.  Indeed,
\begin{align*}
\eta v_{42}+\kappa\nu v_{45}
  &=\nu\eta(W+\kappa U)+\kappa\nu\xi e_1-\eta\mu e_4,\\
\xi v_{12}+\kappa\mu v_{15}
  &=\mu\xi(W+\kappa U)-\kappa\mu\eta e_1+\xi\nu e_4.
\end{align*}
Using~\eqref{eq:key-WU}, the $e_0$- and $e_2$-coordinates in these two
expressions vanish.  Therefore
\begin{align}
v_{13}&=\frac{\kappa\nu\xi e_1+(N/\varphi)\nu\eta e_3
                 -\eta\mu e_4}{r_{13}},\label{eq:v13coords}\\
v_{43}&=\frac{-\kappa\mu\eta e_1+(N/\varphi)\mu\xi e_3
                 +\xi\nu e_4}{r_{43}}.\label{eq:v43coords}
\end{align}
Since $N^2/\varphi^2=1+\kappa^2$, the unit-circle identities for the
parameters give
\[
\|v_{13}\|=\|v_{43}\|=1,
\qquad
\ip{v_{13}}{v_{43}}=0.
\]
We now complete the orthonormal basis in column~$3$.  Identify
$\spn\{e_1,e_3,e_4\}$ with oriented $\R^3$ in the ordered basis
$(e_1,e_3,e_4)$, and define
\begin{equation}\label{eq:v53}
v_{53}=v_{13}\times v_{43}.
\end{equation}
Since $v_{13}$ and $v_{43}$ are orthonormal, their cross product has unit norm
and is perpendicular to both.  Thus $(v_{13},v_{43},v_{53})$ is an
orthonormal basis of $\spn\{e_1,e_3,e_4\}$.

We shall use the following change of basis twice.  Let $(x,y)$ be a real
orthonormal pair, and let $z$ have a nonzero projection onto
$\spn\{x,y\}$.  Define
\begin{equation}\label{eq:completion}
\mathcal C_z(x,y)=
\left(
\frac{\ip{z}{y}x-\ip{z}{x}y}{\rho},
\frac{\ip{z}{x}x+\ip{z}{y}y}{\rho}
\right),
\qquad
\rho=\sqrt{\ip{z}{x}^2+\ip{z}{y}^2}.
\end{equation}
If the projection of $z$ has coordinates $(a,b)$ in the basis $(x,y)$, then
the second vector in~\eqref{eq:completion} is the normalized projection and
the first has coordinates proportional to $(b,-a)$.  Hence
$\mathcal C_z(x,y)$ is an orthonormal pair in $\spn\{x,y\}$ whose first
member is perpendicular to $z$.  We apply this construction by setting
\begin{equation}\label{eq:v2021v3031}
(v_{20},v_{21})=\mathcal C_{v_{10}}(Q,S),
\qquad
(v_{30},v_{31})=\mathcal C_{v_{10}}(P,R).
\end{equation}
Write $\rho_{20}$ and $\rho_{30}$ for the two corresponding normalizing
factors in~\eqref{eq:completion}.  Direct substitution of the parameters
in~\eqref{eq:parameters} gives $\rho_{20}>0$ and $\rho_{30}>0$, so both pairs
are well defined.

We next define the two remaining entries in row~$4$.  They must form an
orthonormal basis of $\spn\{Y,T\}$, with $v_{41}$ perpendicular to $e_0$.
Since $\ip{e_0}{Y}=\alpha$, put
\begin{equation}\label{eq:v4041}
\tau=\ip{e_0}{T},\qquad
r_{40}=\sqrt{\alpha^2+\tau^2},\qquad
v_{40}=\frac{\alpha Y+\tau T}{r_{40}},\qquad
v_{41}=\frac{\tau Y-\alpha T}{r_{40}}.
\end{equation}
Since $r_{40}\geq\alpha>0$, these vectors are well defined.

\subsection{The compatibility condition}

It remains to choose the two entries $v_{50}$ and $v_{51}$.  Let
\[
L=\spn\{v_{13},v_{43}\}.
\]
For column~$0$, we need a unit vector in $L$ which is perpendicular to both
$v_{20}$ and $v_{30}$.  Such a vector exists precisely when the projections
of $v_{20}$ and $v_{30}$ onto $L$ are linearly dependent.  Write the
coordinates of these projections in the basis $(v_{13},v_{43})$ as
\begin{equation}\label{eq:aabb}
a_1=\ip{v_{20}}{v_{13}},\quad
a_2=\ip{v_{20}}{v_{43}},\quad
b_1=\ip{v_{30}}{v_{13}},\quad
b_2=\ip{v_{30}}{v_{43}}.
\end{equation}
Since $(v_{13},v_{43})$ is orthonormal, the two projections are linearly
dependent if and only if
\begin{equation}\label{eq:compatibility}
a_1b_2-a_2b_1=0.
\end{equation}
This condition constrains the rotation parameters.  For the rational values
in~\eqref{eq:parameters}, direct substitution and clearing denominators give
\begin{align}
\Delta={}&\varphi^2\gamma^2\xi^2\eta^2
+2\alpha\mu\delta\varepsilon\varphi\gamma\nu\xi\eta
-\alpha^2\varepsilon^2\gamma^2\mu^2\nu^2
-\beta^2\varepsilon^2\varphi^2\nu^2\xi^2=0.\label{eq:Delta}
\end{align}
Using the five unit-circle identities, the determinant
in~\eqref{eq:compatibility} factors as
\begin{equation}\label{eq:compatibility-factor}
a_1b_2-a_2b_1
=\frac{\alpha\Delta}
{\varphi^2r_{13}r_{43}\rho_{20}\rho_{30}}.
\end{equation}
The denominator is nonzero, and therefore~\eqref{eq:compatibility} holds.
Thus the projections are linearly dependent.  Moreover,
$a_2v_{13}-a_1v_{43}$ is perpendicular to both projections, and hence to
$v_{20}$ and $v_{30}$.  We normalize this vector and take its orthogonal
companion by defining

\begin{equation}\label{eq:v5051}
r_{50}=\sqrt{a_1^2+a_2^2},\qquad
v_{50}=\frac{a_2v_{13}-a_1v_{43}}{r_{50}},\qquad
v_{51}=\frac{a_1v_{13}+a_2v_{43}}{r_{50}}.
\end{equation}
Here $a_1^2+a_2^2=21025/398129>0$, so the normalization is well defined.
Therefore $(v_{50},v_{51})$ is an orthonormal basis of $L$, and $v_{50}$ is
perpendicular to both $v_{20}$ and $v_{30}$.

\subsection{The punctured array}

We have now defined every entry of the array $V_{29}$ in~\eqref{eq:V29}.

\begin{proposition}\label{prop:V29}
The array $V_{29}$ is a punctured orthonormal array of order six in $\R^5$.
\end{proposition}

\begin{proof}
We verify the punctured rows first.  We then verify columns $0,2,3,4,5$ and
use Lemma~\ref{lem:missing-column} for the remaining column.

\emph{Rows.}
Row~$0$ is the standard basis.  In row~$1$, $Y$ is perpendicular
to $\spn\{W,e_4,U,e_1\}$.  This four-dimensional space is the orthogonal
direct sum of $\spn\{v_{12},v_{15}\}$ and
$\spn\{v_{42},v_{45}\}$.  The pair $(v_{12},v_{15})$ is an orthonormal basis
of the first subspace, while $(v_{13},v_{10})$ is an orthonormal basis of the
second.  Hence row~$1$ is an orthonormal basis.

For row~$2$, equation~\eqref{eq:completion} replaces the pair $(Q,S)$ in the
orthonormal basis $(e_0,P,R,Q,S)$ by the orthonormal pair $(v_{20},v_{21})$.
Similarly, in row~$3$ it replaces $(P,R)$ in the orthonormal basis
$(P,R,U,Q,H)$ by $(v_{30},v_{31})$.  Thus rows $2$ and $3$ are orthonormal.
For row~$4$, the pair $(Y,T)$ is perpendicular to
$(v_{42},v_{43},v_{45})$, and~\eqref{eq:v4041} replaces $(Y,T)$ by the
orthonormal pair $(v_{40},v_{41})$.  Finally,~\eqref{eq:v5051} replaces
$(v_{13},v_{43})$ by $(v_{50},v_{51})$.  It follows that
$(v_{50},v_{51},v_{53},Y,Z)$ is an orthonormal basis, which proves the claim
for row~$5$.

\emph{Columns $2,3,4,5$.}
Columns $2,3,4,5$ are respectively
\begin{align*}
&(e_1,v_{12},U,v_{42},Y),\\
&(e_2,v_{13},e_0,v_{43},v_{53}),\\
&(e_3,Y,P,Q,Z),\\
&(e_4,v_{15},R,H,v_{45}).
\end{align*}
Each is an orthonormal basis by the orthogonal decompositions used in the
construction.

\emph{Column $0$.}
This is the column for which the compatibility condition is needed.  All five
vectors in this column have unit norm.  The completion rule gives
$v_{10}\perp v_{20},v_{30}$, and
$\spn\{Q,S\}\perp\spn\{P,R\}$ gives $v_{20}\perp v_{30}$.  Moreover,
$v_{10}\perp v_{40},v_{50}$ because
$v_{10}\perp\spn\{Y,T,v_{13},v_{43}\}$, while
$v_{40}\perp v_{50}$ because
$\spn\{Y,T\}\perp\spn\{v_{13},v_{43}\}$.  It remains to verify the pairs
involving $v_{20}$ or $v_{30}$ with $v_{40}$ or $v_{50}$.

Put $q=\ip{v_{10}}{Q}$ and $p=\ip{v_{10}}{P}$.  Direct substitution, using
$N^2=1-\beta^2\varepsilon^2$, gives
\begin{align*}
\tau&=-\frac{\beta\varphi r_{43}}{N},&
\ip{v_{20}}{Y}&=-\frac{\beta\varphi q}{\rho_{20}},&
\ip{v_{20}}{T}&=-\frac{\alpha Nq}{r_{43}\rho_{20}},\\
\ip{v_{30}}{Y}&=-\frac{\beta\varepsilon p}{\rho_{30}},&
\ip{v_{30}}{T}&=-\frac{\kappa Np}{r_{43}\rho_{30}}.
\end{align*}
Consequently
\begin{align*}
\alpha\ip{v_{20}}{Y}+\tau\ip{v_{20}}{T}&=0,\\
\alpha\ip{v_{30}}{Y}+\tau\ip{v_{30}}{T}&=0,
\end{align*}
where the second equality uses $\varphi\kappa=\alpha\varepsilon$.  Since
$v_{40}=(\alpha Y+\tau T)/r_{40}$, this proves that both $v_{20}$ and
$v_{30}$ are perpendicular to $v_{40}$.  Finally, the definitions
in~\eqref{eq:aabb} and~\eqref{eq:v5051} give
\[
\ip{v_{20}}{v_{50}}=\frac{a_2a_1-a_1a_2}{r_{50}}=0,
\qquad
\ip{v_{30}}{v_{50}}=\frac{a_2b_1-a_1b_2}{r_{50}}=0,
\]
where the second equality uses~\eqref{eq:compatibility}.  Thus every pair of
distinct vectors in column~$0$ is orthogonal, so this column is an
orthonormal basis.

\emph{Column $1$.}
We have proved that every punctured row and each of the columns
$0,2,3,4,5$ is an orthonormal basis.  Lemma~\ref{lem:missing-column} now gives
the result for column~$1$.
\end{proof}

We now extend the punctured array to a quantum Latin square.  Let
$d=(1,0,0,0,0,0)^T\in\R\oplus\R^5$ and write
$\widehat v=(0,v)^T$ for $v\in\R^5$.  Define
\begin{equation}\label{eq:Phi29}
(\Phi_{29})_{ij}=\begin{cases}
d,&i=j,\\
\widehat{(V_{29})_{ij}},&i\ne j.
\end{cases}
\end{equation}

\begin{proof}[Proof of Theorem~\ref{thm:main}]
By Propositions~\ref{prop:V29} and~\ref{prop:diagonal-extension},
$\Phi_{29}$ is a quantum Latin square.  It remains to determine its
cardinality.

The six diagonal entries all equal $d$.  Exactly two off-diagonal rays occur
twice:
\[
(V_{29})_{01}=(V_{29})_{23}=e_0,
\qquad
(V_{29})_{14}=(V_{29})_{52}=Y.
\]
Table~\ref{tab:separation} shows that every other off-diagonal entry lies on a
different ray.  Hence the thirty off-diagonal entries determine twenty-eight
rays: the rays of $e_0$ and $Y$ each occur twice, and the other twenty-six
rays occur once.  The diagonal ray lies in the added orthogonal summand and
is distinct from all of them.  Therefore $\card(\Phi_{29})=29$.
\end{proof}

\section{The order-six spectrum}

\begin{corollary}\label{cor:spectrum}
For order six,
\[
\Spec(\operatorname{QLS}(6))=\{6,8,9,\ldots,36\}.
\]
Equivalently, a quantum Latin square of order six exists with every
cardinality from $6$ through $36$ except $7$.
\end{corollary}

\begin{proof}
Zhang, Lv, and Cao list
$[6,22]\setminus\{7\}$ together with
$24,26,28,30,31,33,34,$ and $36$ as attainable, and list
$23,25,27,29,32,$ and $35$ as unresolved~\cite[Table~6]{ZhangLvCao2026}.
Xu gives constructions for $23,25,27,32,$ and $35$~\cite{Xu2026}.  Thus, of
the values from $6$ through $36$, only $7$ and $29$ remain.  Theorem~\ref{thm:main}
supplies $29$, while the general obstruction to cardinality $n+1$ excludes
$7$~\cite[Lemma~3.1]{ZhangWangJi2026}.  This proves the result.
\end{proof}

The example constructed here is real.  It is obtained by fixing a common
diagonal ray and coordinating orthogonal changes of basis in the
five-dimensional complement.  In contrast, Xu's constructions of
cardinalities $27$, $32$, and $35$ use Hadamard products.  It remains an
interesting question whether cardinality $29$ can also be obtained from a
Hadamard-product construction, or from a construction with simpler and more
symmetric coordinates.

\section*{Computational and AI assistance}

OpenAI Codex assisted with the parameter search, exact symbolic verification,
reference checking, and manuscript revision.  The authors take responsibility
for the mathematical content and the final text.

\appendix

\section{Projective signatures for cardinality 29}\label{app:separation}

We now verify the ray count used in the proof of Theorem~\ref{thm:main}.
Assign labels to the cells of $\Phi_{29}$ as follows:
\begin{equation}\label{eq:C29}
C_{29}=
\begin{pmatrix}
00&01&03&04&05&06\\
07&00&08&09&02&10\\
11&12&00&01&13&14\\
15&16&17&00&18&19\\
20&21&22&23&00&24\\
25&26&02&27&28&00
\end{pmatrix}.
\end{equation}
The label $00$ occurs in the six diagonal cells.  The label $01$ occurs at
$(0,1)$ and $(2,3)$, while the label $02$ occurs at $(1,4)$ and $(5,2)$.
Each label from $03$ through $28$ occurs once.

For $x=(x_0,\ldots,x_4)^T$, let
$\operatorname{supp}(x)=\{k:x_k\ne0\}$.  If two nonzero vectors are
projectively proportional, then they have the same support and the same
ratios between corresponding nonzero coordinates.  We use these invariants
to separate the off-diagonal labels.  Table~\ref{tab:separation} groups the
labels by support and records the required coordinate ratios, with the values
in each row listed in label order.  For support $\{1,2,3,4\}$, we use the
ordered pair
\[
r(x)=\left(x_4/x_1,\;x_2/(\sqrt{481}x_1)\right).
\]

\begin{table}[htbp]
\centering
\caption{Projective signatures for the off-diagonal classes.}
\label{tab:separation}
\scriptsize
\renewcommand{\arraystretch}{1.18}
\begin{tabular}{p{0.22\linewidth} p{0.13\linewidth} p{0.53\linewidth}}
\toprule
Support & Labels & Separating data\\
\midrule
$\{0\},\{1\},\{2\}$ & $01,03,04$ & The supports are unique.\\
$\{3\},\{4\},\{2,3\}$ & $05,06,14$ & The supports are unique.\\
$\{0,2\}$ & $02,28$ &
$x_2/x_0:\ \frac{39}{80},\ -\frac{80}{39}$.\\
$\{0,1,2,3,4\}$ & $07,20$ &
$x_2/x_0:\ -\frac{80}{39},\ \frac{1070067615}{21844238533}$.\\
$\{0,2,3,4\}$ & $08,22$ &
$x_4/(\sqrt{481}x_0):\ -\frac{6853}{10800},\ \frac{267}{1925}$.\\
$\{1,3,4\}$ & $09,23,25,26,27$ &
$x_4/x_1:\ -\frac{10413}{53900},\ -\frac{47971}{1300},\
-\frac{70444321}{63198300},\ \frac{49642153}{57155300},\
\frac{1100}{1157}$.\\
$\{0,1,2,3\}$ & $10,24$ &
$x_1/(\sqrt{481}x_0):\ -\frac{13}{756},\ \frac{28}{39}$.\\
$\{1,2,3,4\}$ & $11,12$ &
$r:\ (-\frac{55}{48},-\frac{136437}{17680}),\
(-\frac{55}{48},\frac{6205}{11686857})$.\\
$\{1,2,3,4\}$ & $15,16$ &
$r:\ (\frac{48}{55},\frac{1695717}{105925820}),\
(\frac{48}{55},-\frac{10220}{301977})$.\\
$\{1,2,3,4\}$ & $21$ &
$r:\ (\frac{1100}{1157},\frac{3583161}{11380460})$.\\
$\{1,4\}$ & $13,18$ &
$x_4/x_1:\ \frac{48}{55},\ -\frac{55}{48}$.\\
$\{0,2,3\}$ & $17,19$ &
$x_2/x_0:\ -\frac{80}{39},\ \frac{27}{65}$.\\
\bottomrule
\end{tabular}
\end{table}

In the first two rows of Table~\ref{tab:separation}, each listed support occurs
only once.  For every other support except $\{1,2,3,4\}$, the displayed
coordinate ratios are distinct.  For support $\{1,2,3,4\}$, the first
component of $r$ separates the three groups in the table, and the second
component separates the two labels in each of the first two groups.  It
follows that labels $01,\ldots,28$ represent twenty-eight distinct
off-diagonal rays.  The label $00$ lies in the added one-dimensional summand
and cannot coincide with any of them.  Hence~\eqref{eq:C29} gives exactly the
ray classes of $\Phi_{29}$.

\end{document}